\documentclass[11pt, leqno,amscd, psamsfonts]{amsart}
\usepackage{amsthm, amsmath}
\usepackage{amssymb} 
\usepackage{amscd}

\DeclareFontEncoding{OT2}{}{} 

\def\arrowdown#1#2{\Big\downarrow \rlap{$\vcenter{\hbox{$\scriptstyle#2$}}$}
{\hbox to -10pt{\hss{$\vcenter{\hbox{$\scriptstyle#1$}}$}}}}
\def\arrowup#1#2{\Big\uparrow \rlap{$\vcenter{\hbox{$\scriptstyle#2$}}$}
{\hbox to -10pt{\hss{$\vcenter{\hbox{$\scriptstyle#1$}}$}}}}

\newtheorem{thm}{Theorem}[section]

\newtheorem{Remark}[thm]{Remark}

\newcounter{ex}[section]

\numberwithin{equation}{section}
\numberwithin{thm}{section}

\def\thfill{\null\nobreak\hfill}

\def\endproof{\thfill\vbox{\hrule
  \hbox{\vrule\hbox to 5pt{\vbox to 5pt{\vfil}\hfil}\vrule}\hrule}}

\begin{document}
\title[An element of order $4$ in the Nottingham group at the prime $2$]{An element of order $4$ in the Nottingham group at the prime $2$}
\author{T. Chinburg}
\address{T.C.: Department of Mathematics\\
University of Pennsylvania
\\Philadelphia, Pennsylvania 19104-6395, U. S. A.
}
\email{ted@math.upenn.edu}
\thanks{Supported by NSF Grant DMS05-00106}
\author{P. Symonds}
\address{P.S.: \\School of Mathematics\\
University of Manchester\\Oxford Road, 
Manchester M13 9PL
\\Manchester M13 9PL
United Kingdom
}
\email{Peter.Symonds@manchester.ac.uk}

\date{\today}

\maketitle

\section{Introduction}
The Nottingham group $\mathcal{N}(k)$ over a field $k$ of characteristic $p > 0$
is the group of continuous automorphisms of the ring $k[[t]]$ which are equal to the
identity modulo $t^2$.   The object of this paper is to construct for $p = 2$ an explicit element of order $4$ in $\mathcal{N}(k)$.  After giving the proof
we will discuss the relation of this construction to some of the literature concerning  elements of order $p^2$
in $\mathcal{N}(k)$.

\begin{thm}
\label{thm:theresult}
Suppose $k$ has characteristic $p = 2$.  An automorphism $t \mapsto \sigma(t)$ of order $4$ of $k[[t]]$ is given
by setting
\begin{eqnarray}
\label{eq:sigmaform}
\sigma(t) &=& t + t^2 + (t^6) + (t^{12} + t^{14}) + (t^{24} + t^{26} + t^{28} + t^{30}) + (t^{48} + \cdots + t^{62}) + \cdots\nonumber\\
&=& t + t^2 + \sum_{j = 0}^\infty \sum_{\ell =0 }^{2^{j}-1} t^{6\cdot 2^j+2 \ell}
\end{eqnarray}
\end{thm}

\begin{proof} Let $A = k[[t,w]]/(w + (1+t)w^2 + t^3)$.  We may define a continuous automorphism
$\sigma$ of $A$ over $k$ by 
\begin{equation}
\label{eq:sigmadef}
\sigma(t) = (t+w)/(1 + t) \quad \mathrm{and} \quad \sigma(w) = w/(1+t).
\end{equation}
It is easily verified that $\sigma$ has order 4. 

Set $v=w(1+t)$ so that $A = k[[t,v]]/(v^2+v+t^3 +t^4)$ and let 
\begin{equation}
\label{eq:vform}
s = \sum_{i = 0}^{\infty} (t^3 + t^4)^{2^i}.
\end{equation}
Then $s^2 + s = t^3 + t^4$, so $v^2+v+t^3 +t^4 = (v+s)(v+s+1)$. The second factor is invertible, so $A=k[[t,v]]/(v+s)$. Thus $v=s \in A$ and $w=s/(1+t)$.
Substituting this into the expression for $\sigma(t)$ in (\ref{eq:sigmadef}) leads to
$$\sigma(t) = \frac{t}{1+t} + \frac{s}{(1+t)^2} =  \frac{t}{1+t} + \frac{t^3 + t^4}{(1+t)^2} + \frac{\sum_{i=1}^\infty (t^3(1+t))^{2^i}}{(1 + t)^2}.$$
This leads to the formula
in the Theorem on setting $j = i-1$ in the last sum on the right.
  \end{proof}

\begin{Remark}
\label{rem:ellrem} {\rm Let $(x:y:z)$ be homogeneous coordinates for the projective
space $\mathbb{P}^2_k$.  We may define an order $4$ automorphism $\sigma$
of $\mathbb{P}^2_k$ by $\sigma(x) = x+y$, $\sigma(y) = y$ and $\sigma(z) = x + z$.
This $\sigma$ stabilizes the curve $E$ with homogeneous
equation  $z^2 y + (z + x)y^2 + x^3 = 0$ as well as the point $Q = (0:0:1)$ on $E$.  
By \cite[Appendix A]{Silverman}, $E$ is a supersingular elliptic curve with origin $Q$ and $j$-invariant $0$.  The completion of the local ring of $E$
at $Q$ is isomorphic to the ring $A \cong k[[t]]$ above when we let $t = x/z$
and $v = (1 + x/z)(y/z)$, and the action of $\sigma$ on $A$ results from
the action of $\sigma$ on $E$.  }
\end{Remark}

We now discuss some literature pertaining to Theorem \ref{thm:theresult}.

 Camina has shown
in \cite{Camina1} that $\mathcal{N}(k)$ contains every countably based pro-$p$ group
as a subgroup \cite{Camina1}, so in particular $\mathcal{N}(k)$ contains every finite $p$-group.
In \cite{Klopsch}, Klopsch shows that representatives for the conjugacy classes of elements
of order $p$ are given by the automorphisms $t \mapsto t(1 - at^m)^{-1/m}$ as $m$ ranges over
positive integers prime to $p$ and $a$ ranges over elements of $k^*$.  In \cite{Camina2},
Camina wrote concerning explicitly described subgroups of $\mathcal{N}(k)$
that ``An element of order $p^2$ is still not known."   Order $p^2$ automorphisms of $k[[t]]$ over 
$k$ have been extensively studied
by Green and Matignon \cite{GM}, and their work contains implicit formulas
for elements  of $\mathcal{N}(k)$ of order $p^2$.  Barnea and Klopsch  mention
in the introduction of \cite{BK} that the subgroups of the Nottingham group they consider
contain elements of arbitrarily large $p$-power order.  
 In \cite{Lubin}, Lubin
uses  formal groups of height $1$ to give an explicit construction iterative construction 
 of an element 
 \begin{equation}
 \label{eq:explicit}
 t \to \sigma(t) = \sum_{i = 1} a_i t^i
 \end{equation} of $\mathcal{N}(k)$ of order $p^n$ for each
 integer $n \ge 2$.  Related constructions of such
elements have been considered recently by Green in \cite{Green}.   

The formula in 
Theorem \ref{thm:theresult} is of interest because it does not require an iterative
procedure to produce the $a_i$ in (\ref{eq:explicit}), and because the height
of the formal group associated to the elliptic curve $E$ in Remark \ref{rem:ellrem}
is $2$.  It would be very interesting if formal groups of height greater than $1$
lead to similar formulas for all primes $p$.

\end{document}